\documentclass[12pt,reqno]{amsart}
\usepackage{amssymb}
\usepackage{amsmath}
\usepackage{amsthm}
\usepackage{eucal}
\usepackage{color}
\usepackage{amsfonts}
\usepackage{enumitem}
\usepackage[pdftex]{graphicx}
\usepackage{url}
\usepackage[T1]{fontenc}
\usepackage{hyperref}
\newcommand{\R}{\mathbb R}

\newtheorem{theorem}{Theorem}[section]

\newtheorem{remark}[theorem]{Remark}

\newtheorem{lemma}[theorem]{Lemma}
\newtheorem{corollary}[theorem]{Corollary}

\numberwithin{equation}{section}
\title[Asymptotics for solutions of the BO equation]{On Local Energy decay for solution of the Benjamin-Ono equation}
\author[R. Freire]{Ricardo Freire}
\address[R. Freire] {IMPA\\ Estrada Dona Castorina 110, Rio de Janeiro 22460-320, RJ Brazil}
\email{rickcar8@impa.br}

\author[F. Linares]{Felipe Linares}
\address[F. Linares] {IMPA\\ Estrada Dona Castorina 110, Rio de Janeiro 22460-320, RJ Brazil}
\email{linares@impa.br}

\author[C. Mu\~noz]{Claudio Mu\~noz}
\address[C. Mu\~noz]{CNRS and Departamento de Ingenier\'ia Matem\'atica DIM-CMM UMI 2807-CNRS \\ Universidad de Chile, Santiago, Chile}
\email{cmunoz@dim.uchile.cl}

\author[G. Ponce]{Gustavo Ponce}
\address[G. Ponce]{Department  of Mathematics\\
University of California\\
Santa Barbara, CA 93106\\
USA.}
\email{ponce@math.ucsb.edu}


\begin{document}

\begin{abstract}
We consider the long time dynamics of large solutions to the Benjamin-Ono equation. Using virial techniques, we describe regions of space where every solution in a suitable Sobolev space must decay to zero along sequences of times. Moreover, in the case of exterior regions, we prove complete decay for any sequence of times. The remaining regions not treated here are essentially the strong dispersion and soliton regions.
\end{abstract}

\maketitle

{\footnotesize
\noindent
{\it Resum\'e}. On consid\`ere la dynamique en temps longs de l'{}\'equation de Benjamin-Ono. En utilisant des techniques du viriel, on d\'ecrit les r\'egions du space o\`u toute solution do\^it d\'ecroire vers z\'ero dans une espace de Sobolev bien choisi, au moins s\^ur une suite de temps. Puis, on montre la d\'ecroissance compl\`ete dans le cas de r\'egions ext\'erieures. Les r\'egions pas trait\'ees ici sont celles avec de la dispersion forte et la region des solitons. % The remaining regions not treated here are essentially the strong dispersion and soliton regions.
%\end{abstract}
}

 \section{Introduction}

 We consider the initial value problem (IVP) associated to the Benjamin-Ono (BO) equation
\begin{equation}\label{BO}
\begin{cases}
\partial_tu - \mathcal{H}\partial^2_xu + u\partial_xu = 0, \hskip10pt x,t\in\R,\\
u(x,0)=u_0(x)
\end{cases}
\end{equation}
where $u = u(x,t) $ is a real-valued function and $\mathcal{H}$ is the Hilbert transform, defined on the line as
\begin{equation}
 \mathcal{H}f(x) = {\rm p.v.}\frac{1}{\pi}\int_{\mathbb{R}}\frac{f(y)}{x-y}dy.
\end{equation}
The BO equation was first deduced in the context of long internal gravity waves in a stratified fluid \cite{Ben,On}. Later, the BO equation was shown to be completely integrable  (see \cite{AblowitzSegur} and references therein).

In particular, it possesses an infinite number of conservation laws, being  the first  three the following : 
\begin{equation}
  \begin{split}
  \label{claw}
  I_1(u) &= \int_{\mathbb{R}} u\, dx,\\
 I_2(u)&=  M(u) = \int_{\mathbb{R}}u^2\,dx,\\
   I_3(u)&=E(u) = \int_{\mathbb{R}}\left(\frac{1}{2}|D^{1/2}u|^2 + \frac{1}{6}u^3  \right) \,dx ,
\end{split}
\end{equation}
where $\widehat{D^sf}(\xi)=|\xi|^s\,\widehat{f}(\xi)$. 

 The $k$-conservation law, $I_k(\cdot), \,k\geq 2,$ provides a global in time \it a priori \rm estimate of the norm $\,\|D^{(k-2)/2}u(t)\|_{L^2}$ of the solution $u=u(x,t)$ of the  \eqref{BO}. 

The IVP \eqref{BO} has been extensively studied, especially, the local well-posedness (LWP) and global well-posedness (GWP) measured in the Sobolev scale $H^s(\mathbb{R})=(1-\partial^2_x)^{-s/2}L^2(\mathbb{R})$, $s \in \mathbb{R}$.  In this regard,  one has the following list of  works: Iorio \cite{Iorio}, Abdelouhab et al. \cite{Bona}, Ponce \cite{Poncegwp},  Koch-Tzvetkov \cite{KochTzvetkov}, Kenig and Koenig \cite{kenigkoenig}, Tao \cite{Tao}, Burq-Planchon \cite{BurqPlanchon}, Ionescu-Kenig \cite{I-Ke}, Molinet-Pilod \cite{MolinetPilod} and Ifrim-Tataru \cite{IfTat}, among others. In particular, in \cite{I-Ke} the global well-posedness in $L^2(\R)$ of the IVP \eqref{BO} was established. For further details and results concerning the IVP associated to the BO equation we refer to Saut \cite{Sa19}.

It should be pointed out that in \cite{MolinetSautTzvetkov} it was proved  that none well-posedness for the IVP \eqref{BO} in $H^s(\mathbb{R})$ for any $ s \in \mathbb{R}$  can be established by an argument based only the contraction principle argument.

We recall that the BO equation possesses traveling wave solutions (solitons) $u(x,t)=\phi(x-t)$ of the form
\begin{equation}
\label{twBO}
\phi(x)=\frac{4}{1+x^2},
\end{equation}
which is smooth and exhibits a  mild decay.

\medskip

In this work, we are interested in the asymptotic behavior of solutions to the IVP \eqref{BO}. In fact, we shall deduce some decay properties for solutions of \eqref{BO}  as time evolves.

% We shall mention the recent works that inspired us and introduced some techniques to investigate some properties in this direction as the works of Mu\~{n}oz and Ponce for \eqref{BO} in \cite{MP}. After, Linares,
%Mendez and Ponce used this approach in generalized BO equation in \cite{LMP} and in \cite{Linares2021OnLT} Linares and Mendez used this tecnhiques in the  Schr\"{o}dinger-Korteweg-de %Vries system . To problems in high dimension you can find in the works \cite{mendez2021long} and \cite{MMPP}.

 Our  main results in this work are the following:

\begin{theorem}\label{L2BO}
Let $u_0 \in L^2(\mathbb{R}) $ and  $u = u(x,t)$ be the global in time solution of the IVP \eqref{BO} such that 
$$u \in C(\mathbb{R}:L^2(\mathbb{R}))\cap L^{\infty}(\mathbb{R}: L^2(\mathbb{R})).$$
Then 
  \begin{equation}\label{EQ1}
   \liminf_{t \to \infty }\int_{B_{t^b}(0)}u^2(x,t)\, dx = 0,
  \end{equation}
 where $B_{t^b}(0)$ denotes the ball centered in the origin with radius $t^b$,
 \begin{equation}
  B_{t^b} (0):= \{x \in \mathbb{R} : |x| < t^b\} \quad \mbox{with}\quad 0 < b < \frac{2}{3}.
 \end{equation} 
 
Moreover, there exist a constant $C > 0$ and an increasing sequence of times $t_n \to \infty $ such that 
    \begin{equation}\label{EQ2}
    \int_{B_{t_n^b}(0)}u^2(x,t_n)\, dx \leq \frac{C}{\log^{\frac{(1-b)}{b}}(t_n)}.
    \end{equation}

\end{theorem}

As a consequence of the proof of this theorem we have:

\begin{corollary}\label{corL2BO} Let $u_0 \in L^2(\mathbb{R}) $ and  $u = u(x,t)$ be the global in time solution of the IVP \eqref{BO} such that 
$$u \in C(\mathbb{R}:L^2(\mathbb{R}))\cap L^{\infty}(\mathbb{R}: L^2(\mathbb{R})).$$
Then 
 \begin{equation}\label{notcentered}
   \liminf_{t \to \infty }\int_{B_{t^b}(t^m)}u^2(x,t)\, dx = 0,
  \end{equation}
  where
  \begin{equation}
 B_{t^b}(t^m) := \{ x \in \mathbb{R} :|x - t^m|< t^b\},
\end{equation}
with
\begin{equation}\label{conditions-on-m}
0 < b < \frac{2}{3} {\hskip15pt\text{and} \hskip15pt}    0  \leq m < 1 - \frac{3}{2}b.
\end{equation}
\end{corollary}

\medskip

\begin{remark}\label{rem1}

Under the additional hypothesis:
\begin{equation}
\label{add-hyp}
\begin{aligned}
\text{There exist } &\;a\in [0,1/2)\;\text{and} \; c_0>0 \;\text{such that for all }\;T>0\\
%&\exists \,a\in [0,1/2)\; \;\exists \,c_0>0 \,\;s. t.\, \;\forall \,T>0\\
&\sup_{t\in[0,T]}\int_{-\infty}^{\infty} |u(x,t)|\,dx\leq c_0\,(1+T^2)^{a/2},
\end{aligned}
\end{equation}
a related result to those in Theorem \ref{L2BO} and Corollary \ref{corL2BO} was established in \cite{MP}. The argument of the proof in \cite{MP} was based on virial identities (or weighted energy estimate) first appearing in \cite{MP1} in the study of the long time behavior of solution of the generalized Korteweg-de Vries (KdV) equation. In \cite{MPS} and  \cite{LMP} this was extended, adapted  and generalized to others one dimensional dispersive nonlinear systems  under an assumption similar to that in \eqref{add-hyp}.

 In \cite{MMPP} a key idea was introduced to remove the hypothesis \eqref{add-hyp} and to extend the argument to higher dimensional dispersive model. This approach was further implemented in  \cite{mendez2021long} and \cite{Linares2021OnLT} for systems.

\end{remark}

Next, we present a result concerning the decay of solutions in the energy space :

\begin{theorem}\label{H1BO}
  Let $u_0 \in H^{1/2}(\mathbb{R}) $ and $u = u(x,t)$ be the global in time solution of the  IVP \eqref{BO} such that 
  $$u \in C(\mathbb{R}:H^{1/2}(\mathbb{R}))\cap L^{\infty}(\mathbb{R}: H^{1/2}(\mathbb{R})).$$
   Then 
  \begin{equation}\label{EQ3}
   \liminf_{t \to \infty }\int_{B_{t^b}(0)}\left(u^2(x,t)+|D_x^{1/2}u(x,t)|^2\right)\, dx = 0,\;\;\;\;0<b<\frac{2}{3}.
  \end{equation}
\end{theorem}

Now, we consider the asymptotic decay of the solution in a domain moving in time in the right direction:

\begin{theorem}
\label{to-the-right}
There exists a constant $C_0>0$ depending only on $\|u_0\|_{H^1}$ such that the global  solution 
\[
u\in C(\mathbb R:H^1(\mathbb R))\cap L^{\infty}(\mathbb R:H^1(\mathbb R))
\]
of IVP \eqref{BO} satisfies
\begin{equation}
\label{main}
\lim_{t\to \infty} \| u(t)\|_{L^2(x\geq C_0t)}=0.
\end{equation}
\end{theorem}

\begin{remark} \label{abc}\hskip10pt

\begin{enumerate}
\item This result is inspired in a similar one found in \cite{MaMe} for the generalized KdV equation. In fact, the proof in \cite{MaMe} is a generalization of the identity used in \cite{Ka} to establish the so called Kato local smoothing effect in solutions of the generalized KdV equation.   The proof for the KdV  is significantly simpler. In the case of the BO equation the proof follows  the virial identity obtained in Lemma \ref{virial-BO}, \eqref{id1}, and some commutator estimates, see the comments in Remark \ref{rem-imp} below.
\item

From the scaling argument, i.e. if $u(x,t)$ is a solution of the BO equation, then for any $\lambda>0$, $\,u_{\lambda}(x,t)=\lambda u(\lambda x,\lambda^2 t)$ is also a solution, one sees that for any $c>0$ one has a traveling wave solution (soliton), see \eqref{twBO},
$$
u_{c}(x,t)=c\, \phi (cx-c^2t)=c\,\phi(c(x-ct)).
$$
The speed of propagation of the soliton $c$ is proportional to its amplitude$\;c\,\| \phi\|_{\infty}$. This is consistent with the statement in \eqref{main}.

\item 
Combining \eqref{main} and the conservation laws of the BO equation and under the same hypothesis on Theorem 
\ref{to-the-right} one gets that for any $p\in(2,\infty]$ and any $C_0'>C_0$ 
\begin{equation}
\label{main-p}
\lim_{t\to \infty} \| u(t)\|_{L^p(x\geq C_0't)}=0
\end{equation}
and for any $s\in (0,1)$
\begin{equation}
\label{main-s}
\lim_{t\to \infty} \left\| D_x^s\left(u(x,t)\,\chi\left(\frac{x}{2C_0'}\right)\right)\right\|_{L^2}=0,
\end{equation}
with $\chi\in C^\infty(\R)$, $ 0\leq \chi(x) \leq 1\;$ for all $\;x\in\R$,
$\chi(x) \equiv 0\;$ if $\;x\leq 1$, $\chi(x) \equiv 1\;$ if $\;x\geq 2$ and $\;\chi'\geq 0$.

\item 
If, in addition, one assumes that the global  solution  $u=u(x,t)$ satisfies
 $$u\in C(\mathbb R:H^{3/2}(\mathbb R))\cap L^{\infty}(\mathbb R:H^{3/2}(\mathbb R)),$$ then
\begin{equation}
\label{main-1}
\lim_{t\to \infty} \| \partial_xu(t)\|_{L^2(x\geq C_0't)}=0,
\end{equation}
and \eqref{main-s} holds for $s\in (0,3/2)$. This result extends to global  solutions   $u\in C(\mathbb R:H^{k/2}(\mathbb R))$, $k\in\mathbb N$ with $k>3$.

\end{enumerate}
\end{remark}

The next result studies the decay of the $L^2$-norm of the solution in the far left region:

\begin{theorem}
\label{to-the-left}
For any constant $C_1>0$ and any $\eta>0$ the global in time solution 
\[
u\in C(\mathbb R:H^1(\mathbb R))\cap L^{\infty}(\mathbb R:H^1(\mathbb R))
\]
 of IVP \eqref{BO} satisfies
\begin{equation}
\label{main2}
\lim_{t\to \infty} \| u(t)\|_{L^2(x\leq -C_1t \log^{1+\eta}t)}=0.
\end{equation}
\end{theorem}

\begin{remark} \hskip10pt

\begin{enumerate}

\item To our knowledge the result in Theorem \ref{to-the-left} is totally new. From its proof below it will be clear that it applies with minor modifications to solutions of the generalized KdV equation,
$$
\partial_tu +\partial^3_xu + u^k\partial_xu = 0, \hskip5pt x,t\in\R,\;k=1,2,\dots,
$$
 to solutions of the generalized BO equation (see \eqref{k-gBO}), and solutions to others one dimensional dispersive models.

\item 

The statements  \eqref{main-p}, \eqref{main-s} and \eqref{main-1} in Remark \ref{abc} apply to the result in Theorem \ref{to-the-left} with the appropriate modifications.

\item 
Collecting the information in Theorems \ref{L2BO}, \ref{to-the-right} and \ref{to-the-left} one can deduce several estimates. In particular, one has: 
there exists $\,C_0=C_0(\|u_0\|_{H^1})>0$ and an increasing sequence of times $(t_n)_{n=1}^{\infty}$ with $t_n\uparrow \infty$ as $n\to \infty$ such that for any 
constants $c>0,\,\gamma>0$,
\begin{equation}
\label{123}
\liminf_{n\to \infty} \,\int_{\Omega(t_n)} |u(x,t_n)|^2\,dx=\|u_0\|_2^2,
\end{equation}
with
\[
\Omega(t): =\{x\in\R : -c\,t\,\log^{1+\gamma}t<x<-c\,t^{\frac23^{-}}\hskip5pt \text{or}\hskip7pt c\,t^{\frac23^{-}}\!\!\!<x<C_0\,t\}.
\]
\end{enumerate}
\end{remark}

Finally, we shall consider the possible extensions of the above results to solutions of the IVP associated to the $k$-generalized BO (k-gBO) equation
\begin{equation}\label{k-gBO}
\begin{cases}
\partial_tu - \mathcal{H}\partial^2_xu + u^k\partial_xu = 0, \hskip5pt x,t\in\R,\;k=2,3,...\\
u(x,0)=u_0(x).
\end{cases}
\end{equation}

In this case, the equations in \eqref{k-gBO} are not completely integrable and satisfy (in general) only three conservation laws : $I_1(u)$, $I_2(u)$ in \eqref{claw} and
$$
I_3(u)= \int_{\mathbb{R}}\left(\frac{1}{2}|D^{1/2}u|^2 + \frac{u^{k+1}}{(k+1)(k+2)}  \right) \,dx.
$$

The scaling argument, see Remark \ref{abc}, says : if $u(x,t)$ is a solution of the k-gBO equation in \eqref{k-gBO}, then for any $\lambda>0$, $\,u_{\lambda}(x,t)=\lambda^{1/k} u(\lambda x,\lambda^2 t)$ is also a solution. This suggests  that the critical Sobolev space for the well-posedness  should be $H^{s_k}(\R)$ with $s_k=1/2-1/k$.

The results considered here are concerned with global solutions of the \eqref{k-gBO}. Thus, for the cases $k\geq 2$ these are only known under appropriate smallness assumptions on the data. More precisely : if $k=2$ local well-posedness in $H^{1/2}(\R)$ was established in \cite{KeTa}. This local result extends globally in time if one assumes  that the $L^2$-norm of the initial data $u_0$ is small enough (the blow up result in \cite{MaPi} shows that this restriction is necessary). 

In \cite{Ve}, local well posedness was proven in $H^s(\R)$ for  $s>1/3$ if $k=3$, and for $s\geq s_k=1/2-1/k$ if $k\geq 4$. These local results extend to global ones under a smallness assumption of the $H^{1/2}$-norm of the initial data $u_0$ (see \cite{FaLP}). In all these global results one only has an \it{a priori }\rm  bound of the $H^{1/2}$-norm of the solution.

Our argument of proof of Theorem \ref{to-the-right} depends on a global bound of the $L^{\infty}$-norm of the solution. Hence, the proof of  Theorem \ref{to-the-right} provided below does not extend to these small global solutions of the IVP \eqref{k-gBO}.

The approach to obtain Theorem \ref{to-the-left} only requires a global bound of the $L^{k+2}$-norm of the solution, which follows from that of the $H^{1/2}$-norm. Hence, the result in Theorem \ref{to-the-left} expands to all small global solutions of the IVP \eqref{k-gBO} commented above.

\begin{remark}
In the cases when $k\geq 1$ is {\bf odd}, the arguments utilized to prove  Theorem \ref{L2BO}, Corollary \ref{corL2BO} and Theorem \ref{H1BO} apply to get the results in \eqref{EQ1}, \eqref{notcentered} and \eqref{EQ3} with the term  $u^2$ in the integrand substituted by $u^{k+1}$. However, in this case $k\geq 2$ and small data (in a weighted space),  stronger asymptotic results were accomplished in \cite{HaNa}.
\end{remark}

\medskip

The rest of this paper is ordered as follows: Section \ref{Sect:2} contains the statements of some general estimates to be used in the proofs of the main results. Theorem \ref{L2BO} and Corollary \ref{corL2BO} will be proven in Section \ref{Sect:3}. Section \ref{Sect:4} involves the proof of Theorem \ref{H1BO} and Section \ref{Sect:5} those of Theorem \ref{to-the-right} and Theorem \ref{to-the-left}. Appendix \ref{Sect:6} consists of the proof a commutator estimate stated in Section \ref{Sect:2} and used in Sections \ref{Sect:3}-\ref{Sect:5}.

\medskip

%\begin{remark}
%It is worth noticing that Theorems\ref{L2BO} and \ref{H1BO}  can be extended to the non-centered case if we make some straightforward modifications. This result is still true if we replace $ B_{t^b} $ by 
%the set 
%\begin{equation}
% \tilde{B}_{t^b}^m := \{ x \in \mathbb{R}; |x - t^m|< t^b\},
%\end{equation}
%where
%\begin{equation}
%0 < b < \frac{2}{3} {\hskip15pt\text{and} \hskip15pt}    0  \leq m < 1 - \frac{3}{2}b.
%\end{equation}
%\end{remark}

\section{Preliminaries}\label{Sect:2}

We present a series of estimates we will employ in the proof of our results. 

\begin{lemma}\label{CC}
For any $k,m \in \mathbb N\cup\{0\},\,k+m\geq 1$, and any $p\in(1,\infty)$
 \begin{equation} \label{comm}
 \| \partial_x^k \big[\mathcal H;a\big]\partial_x^mf\|_p\leq c_{p,k,m}\|\partial_x^{k+m}a\|_{\infty} \|f\|_p.
 \end{equation}
 \end{lemma}
 The case $k+m=1$ corresponds to the first Calder\'on commutator estimate \cite{Ca}. The general case of \eqref{comm} was established in \cite{BCo}. For a different proof see \cite{DaGaPo}.

The next estimate is an inequality of Gagliardo-Nirenberg type whose proof can be found in  \cite{BL}.

\begin{lemma}\label{GNSinequality} There exists $C>0$ such that for any $f\in H^{1/2}(\R)$ 
 \begin{equation}
 \|f\|_{L^3} \leq C\|f\|_{L^2}^{\frac{2}{3}}\|D^{1/2}f\|_{L^2}^{\frac{1}{3}}. 
 \end{equation}
\end{lemma}

Along the proof of our results we shall use the following general version of the Leibniz rule for fractional derivatives:

\begin{lemma}\label{leibnizhomog.}  
Let $r\in [1,\infty]$  and $ p_1,p_2, q_1, q_2\in (1,\infty]$ with 
\begin{equation}
\label{LR4}
\frac 1r=\frac1{p_1}+\frac1{q_1}=\frac1{p_2}+\frac1{q_2}.
\end{equation}
 Given $s>0$ there exists $c=c(n,s,r,p_1,p_2,q_1,q_2)>0$ such that for all $f, g\in \mathcal S(\mathbb R^n)$
one has
\begin{equation}
\label{LR2}
\| D^s(fg)\|_r\leq c\left(\|f\|_{p_1}\| D^sg\|_{q_1}+\|g\|_{p_2}\|D^sf\|_{q_2}\right).
\end{equation}
\end{lemma}

For the proof of Lemma \ref{leibnizhomog.} we refer to \cite{GO}. The case $r=p_1=p_2=q_1=q_2=\infty$ was established in \cite{BoLi}, see also \cite{GMN}. For earlier versions of this result see \cite{KaPo} and \cite{KPV93}.

 Finally, we consider a commutator estimate whose proof will be given in the appendix.
   \begin{lemma} \label{comm2}
   
   Let $a\in C^2(\mathbb R)$ with $a',\,a''\in L^{\infty}(\mathbb R)$. There exists $c>0$ such that for all $f\in L^2(\mathbb R)$  
  \begin{equation}
  \label{comm-d-1/2}
% \begin{aligned}
 \| D^{1/2} \big [D^{1/2}; a\big] f\|_{L^2}\leq c\| \widehat{a'}\|_{L^1}\|f\|_{L^2}\\
 \leq c\| a'\|_{L^2}^{1/2} \|a''\|_{L^2}^{1/2}\|f\|_{L^2}.
% \end{aligned}
\end{equation}
\end{lemma}

%%%%%%%%%%%%%%%%%%%%%%%%%%%%%%%%%%%%%%%%%%%%%%%%%%%%%%%%%%%%%%%%%%%%%%%%%%%%%%%%%%%%%%%%%%%%%%%%%%%%%%%%%%%%%%%%%%%%%%%%%%%%%%%%%%%%%%%%%%%%%%%%%%%%%%%%%%%%%%%%%%%%%%%%%%%%%%%%%%%%%%%%%%%%%%%%%%%%%%%%%%%%%%%%%%%%%%%%%%%%%%%%%%%%%%%%%%%%%%%%%%%%%%%%%%%%%%%%%%%%%%%%%%%%%%%%%%%%%%%%%%%%%%%%%%%%%%%%%%%%%%%%%%%%%%%%%%%%%%%%%%%%%%%%%%%%%%%%%%%%%%%%%%%%%%%%%%%%%%%%%%%%

\section{Proof of Theorem \ref{L2BO} and Corollary \ref{corL2BO}}\label{Sect:3}

First we will introduce some notation and definitions.

Let $\phi$ be a smooth even and positive function such that 
\begin{equation}\label{weight}
\begin{cases}
\text{\;\;\rm i)}&\!\! \phi'(x)\leq 0, \;\text{for}\; x\ge 0, \\
\text{\;\rm ii)}&\!\! \phi(x)\equiv 1, \;\text{for}\; 0 \le x \le 1, \phi(x) = e^{-x}  \;\text{for}\; {\color{red} x\ge 2}, \\
&\text{and}\;\; e^{-x} \leq \phi(x) \leq 3e^{-x}  \;\text{for}\; x\ge 0,\\
\text{\rm iii)}&\!\! |\phi'(x)| \leq c\,\phi(x) \;\text{and}\; |\phi''(x)| \leq c\,\phi(x)\\
 &\;\text{for some positive constant} \;c.
\end{cases}
\end{equation}

Let $\psi(x) = \int_0^x \phi (s)ds$. In particular, $|\psi (x)| \leq 1 + 3\int_1^{\infty}e^{-t}dt < \infty$.   \\

Next, we consider a smooth cut-off function $\zeta :\mathbb{R} \rightarrow \mathbb{R}$ such that 
\begin{equation}
 \zeta \equiv 1 {\hskip5pt on \hskip5pt} [0,1], {\hskip10pt} 0 \leq  \zeta\leq 1 {\hskip10pt and \hskip10pt} \zeta \equiv 0 {\hskip5pt on \hskip5pt} (\infty,-1]\cup[2,\infty),
\end{equation}
and define $\zeta_n(x) := \zeta(x-n)$.
 
For the parameters $\delta, \sigma \in \mathbb{R}^+$,  we define 
\begin{equation*}
\phi_{\delta} = \delta\phi\left( \frac{x}{\delta}\right) {\hskip 10 pt \text{and} \hskip10pt}  \psi_{\sigma}(x) = \sigma\psi\left( \frac{x}{\sigma}\right).
\end{equation*}

\medskip

The proof of Theorem \ref{L2BO} will be deduced as a consequence of the  following lemmas, which we shall prove below.

\medskip

First, we start by considering some useful parameters involved in our argument of proof. 
 \begin{equation}\label{parameters}
 \rho(t) = \pm t^m, {\hskip5pt  \hskip5pt}  \mu_1(t) = \frac{t^b}{\log t} {\hskip10pt \hbox{and} \hskip10pt} \mu(t) = t^{(1-b)}\log^2t,
 \end{equation}
where $m$ and $b$ are positive constants satisfying the relations 
\begin{equation}\label{bmrelation}
  0 \leq m \leq 1 - \frac{b}{2}{\hskip10pt \hbox{and} \hskip10pt} 0< b \leq \min\left\{\frac{2}{3}, \frac{2}{2+q}\right\},\;\;\;\;\;\;q>0.
\end{equation}

Since
\begin{equation*}
  \frac{\mu'_1(t)}{\mu_1(t)} = \frac{b}{t} - \frac{1}{t\log t}{\hskip10pt \text{and} \hskip10pt} \frac{\mu'(t)}{\mu(t)} = \frac{(1-b)}{t} + \frac{2}{t\log t}
\end{equation*}
it readily follows that
    \begin{equation}\label{Relationderivatives}
        \frac{\mu'_1(t)}{\mu_1(t)} \sim \frac{\mu'(t)}{\mu(t)} = O\left(\frac1t\right), {\hskip10pt for \hskip10pt} t \gg 1
    \end{equation}
where $t \gg 1$ means the values of t such that $\mu_1'(t)$ is positive. In particular, $[10,+\infty) \subset \{t \gg 1\}$.

%%%%%%%%%%%%%%%%%%%%%%%%%%%%%%%%%%%%%%%%%%%%%%%%%%%%%%%%%%%%%%%%%%%%%%%%%%%%%%%%%%%%%%%%%%%%%%%%%%%%%%%%%%%%%%%%%%%%%%%%%%%%%%%%%%%%%%%%%%%%%%%%%%%%%%%%%%%%%%%%%%%%%%%%%%%%%%%%%%%%%%%%%%%%%%%%%%%%%%%%%%%%%%%%%%%%%%%%
    
\medskip

%{\color{red}The idea here is to concentrate our efforts in the proof of the non-centered case for solutions of the IVP \eqref{BO}, therefore the Therorem\ref{L2BO} will follow directly. Let us start to define the functional following the idea in the paper \cite{MMPP}.}

For $u=u(x,t)$ a solution of the IVP \eqref{BO} we define the functional\\ 
\begin{equation}\label{NC-1}
     \mathcal{I}(t) := \frac{1}{\mu(t)}\int_{\mathbb{R}} u(x,t)\psi_{\sigma}\left( \frac{x}{\mu_1(t)} \right) \phi_{\delta }\left( \frac{x}{\mu_1^q(t)} \right)\,dx, 
        \end{equation}
for $q > 1$.

\begin{lemma}\label{Functionalbounded}
Let $u(\cdot,t) \in L^2(\mathbb{R})$, $t\gg 1$. The functional $\mathcal{I}(t)$  is well defined and bounded in time.
   
\begin{proof}
 The Cauchy-Schwarz inequality and the definition of the functions $\mu(t)$ and $\mu_1(t)$ imply that 
 \begin{equation}
  \begin{split}
   |\mathcal{I}(t)| &\leq \frac{1}{\mu(t)}\|u(t)\|_{L^2}\left\|\psi_{\sigma}\left( \frac{\cdot}{\mu_1(t)} \right)\right\|_{L^{\infty}}\left\|\phi_{\delta}\left( \frac{\cdot}{\mu_1^q(t)}\right)\right\|_{L^2} \\
   & = \frac{\mu_1^{q/2}(t)}{\mu(t)}\|u(t)\|_{L^2}\|\psi_{\sigma}\|_{L^{\infty}}\|\phi_{\delta}\|_{L^2}\\
   & \lesssim_{\sigma, \delta} \frac{1}{t^{(2-2b-bq)/2}}\frac{1}{\log^{(4+q)/2}(t)}\|u_0\|_{L^2}. \\
  \end{split}
 \end{equation}
%  \begin{equation*}
%    \color{red}
 %   2-2b-bq > 0  \implies b < \frac{2}{2 + q}.
 %   \end{equation*}

Since $b$ satisfies the condition \eqref{bmrelation} we have that
\begin{equation*}
\sup_{t \gg 1}|\mathcal{I}(t)| < \infty.
\end{equation*}   
\end{proof}
\end{lemma}

\begin{lemma}\label{BoundL1}
For any $t \gg 1$,  it holds that
\begin{equation}\label{NC-2}
\frac{1}{\mu_1(t)\mu(t)} \int_{\mathbb{R}}u^2(x,t)\,\psi_{\sigma}'\left( \frac{x}{\mu_1(t)} \right)\phi_{\delta}\left( \frac{x}{\mu_1^q(t)} \right)\,dx \leq 4\frac{d}{dt}\mathcal{I}(t) + h(t), 
\end{equation}
where $h(t)\in L^1(t \gg 1)$.
\end{lemma}
  
\begin{proof}
We have that
  \begin{equation}\label{l42-1}
   \begin{split}
   \frac{d}{dt}\mathcal{I}(t)
   &= \frac{1}{\mu(t)}\int_{\mathbb{R}}\partial_t\left( u\psi_{\sigma}\left( \frac{x}{\mu_1(t)} \right) \phi_{\delta}\left( \frac{x}{\mu_1^q(t)} \right)\right)\,dx\\
   &\hskip10pt - \frac{\mu'(t)}{\mu^2(t)}\int_{\mathbb{R}} u\psi_{\sigma}\left( \frac{x}{\mu_1(t)} \right) \phi_{\delta}\left( \frac{x}{\mu_1^q(t)} \right)\,dx\\
   &= A(t) + B(t).
   \end{split}
  \end{equation}
  
Cauchy-Schwarz inequality and the conservation of mass, $I_2$ in \eqref{claw}, yield
    \begin{equation}\label{l42-2}
     \begin{aligned}
   |B(t)| &\leq \left|\frac{\mu'(t)}{\mu^2(t)}\right| \left\|u(t)\right\|_{L^2}\left\|\psi_{\sigma}\left( \frac{\cdot}{\mu_1(t)} \right)\right\|_{L^{\infty}}\left\|\phi_{\delta}\left( \frac{\cdot}{\mu_1^q(t)} \right)\right\|_{L^2} \\
  % & = \Big|\frac{\mu'(t)}{\mu(t)}\Big| \Big |\frac{\mu_1^{q/2}(t)}{\mu(t)}\Big|\|u(t)\|_{L^2}\|\psi_{\sigma}\|_{L^{\infty}}\|\phi_{\delta}\|_{L^2}\\
   & \lesssim_{\sigma , \delta} \frac{1}{t^{(4-2b-bq)/2}}\frac{1}{\log^{(4+q)/2}t}\|u_0\|_{L^2}.\\
   \end{aligned}
    \end{equation}
Hence $B(t)\in L^1(\{t \gg 1\})$ whenever $b\leq \frac{2}{2 +q}$. We remark that this term is bounded in $\{t \gg 1\}$.

% \begin{equation} \label{l42-1}
%      \sup_{t \gg 1} |\mathcal{I}_2(t)| \lesssim \sup_{t \gg 1} \frac{1}{t^{(4-2b-bq)/2}}\frac{1}{\log^{(4+q)/2}t} < \infty,
%    \end{equation}
%whenever $ b  \leq \frac{2}{2 + q}$ and $q > 0$. Then, $\mathcal{I}_2(t) \in L^1(\{t \gg 1\})$ since $b$ satisfies \eqref{bmrelation}. 

To estimate $A(t)$, we first differentiate in time to write
\begin{equation}\label{TermsEq.}
\begin{split}
A(t)&=  \frac{1}{\mu(t)}\int_{\mathbb{R}} u_t(x,t)\psi_{\sigma}\left( \frac{{x}}{\mu_1(t)} \right) \phi_{\delta}\left( \frac{{x}}{\mu_1^q(t)} \right)\,dx\\
 &\hskip10pt - \frac{\mu_1'(t)}{\mu_1(t)\mu(t)}\int_{\mathbb{R}} u(x,t) \left( \frac{{x}}{\mu_1(t)} \right)\psi_{\sigma}'\left( \frac{{x}}{\mu_1(t)} \right) \phi_{\delta}\left( \frac{{x}}{\mu_1^q(t)} \right)\,dx\\
&\hskip10pt  - \frac{q\mu_1'(t)}{\mu_1(t)\mu(t)}\int_{\mathbb{R}} u(x,t) \psi_{\sigma}\left( \frac{{x}}{\mu_1(t)} \right)\left( \frac{{x}}{\mu_1^q(t)} \right) \phi_{\delta}'\left( \frac{{x}}{\mu_1^q(t)} \right)\,dx\\
& = A_1(t)+A_2(t)+A_3(t).
\end{split}
\end{equation}

Using the equation in \eqref{BO} and integrating by parts yield
\begin{equation} \label{l42-3}
\begin{split}
A_1(t) &=\frac{1}{\mu(t)}\int_{\mathbb{R}}\mathcal{H}u(x,t)\,\partial_x^2\left(\psi_{\sigma}\left(\frac{{x}}{\mu_1(t)} \right) \phi_{\delta}\left(\frac{{x}}{\mu_1^q(t)} \right)\right)\,dx\\
&\hskip15pt + \frac{1}{2\mu(t)\mu_1(t)}\int_{\mathbb{R}} u^2(x,t)\psi_{\sigma}'\left(\frac{{x}}{\mu_1(t)} \right) \phi_{\delta}\left(\frac{{x}}{\mu_1^q(t)} \right)\,dx\\
    &\hskip15pt +\frac{1}{2\mu(t)\mu_1^q(t)}\int_{\mathbb{R}}u^2(x,t) \psi_{\sigma}\left(\frac{{x}}{\mu_1(t)} \right) \phi_{\delta}'\left(\frac{{x}}{\mu_1^q(t)} \right)\,dx\\
    &=: A_{1,1}(t) + A_{1,2}(t) + A_{1, 3}(t).
    \end{split}
\end{equation}

We remark that  $ \, A_{1, 2}(t)$ is the term we want to estimate in \eqref{l42-3}. Then we need to show that the reminder terms
are in $L^1(\{t\gg 1\})$.

 Differentiating with respect to $x$, using the Cauchy-Schwarz inequality, Hilbert's transform properties, the conservation of mass,
 and the definition of $\mu(t)$ and $\mu_1(t)$ we deduce that  
\begin{equation}\label{l42-4}
\begin{split}
|A_{1,1}(t)|&\le\frac{1}{\mu(t)\mu_1^{3/2}(t)}\|u(t)\|_{L^2}\|\psi_{\sigma}''\|_{L^2}\|\phi_{\delta}\|_{L^{\infty}}\\
&\hskip15pt +\frac{1}{\mu(t)\mu_1^{(1+2q)/2}(t)}\|u(t)\|_{L^2}\|\psi_{\sigma}'\|_{L^2}\|\phi_{\delta}'\|_{L^{\infty}}\\
&\hskip15pt+\frac{1}{\mu(t)\mu_1^{3q/2}(t)}\|u(t)\|_{L^2}\|\psi_{\sigma}\|_{L^{\infty}}\|\phi_{\delta}''\|_{L^2}\\
&\lesssim_{\sigma,\delta}  \frac{ \|u_0\|_{L^2}}{t^{(2 + b)/2}\log^{1/2}t}+\frac{ \|u_0\|_{L^2}}{t^{(2 - b + 2bq)/2}\log^{(\frac{3}{2}-q)}(t)} \\
&\hskip15pt +\frac{ \|u_0\|_{L^2}}{t^{(2 - 2b + 3bq)/2}\log^{(4-3q)/2}t}.
\end{split} 
\end{equation}

Since $q>1$ it follows  that $A_{1,1}(t)\in L^1(\{t\gg1\})$.

The term $A_{1,3}$ can be bounded by employing the conservation of mass, and the definition of 
$\mu(t)$ and $\mu_1(t)$.
\begin{equation}\label{l42-9}
\begin{split}
 |A_{1,3}| &\leq  \frac{\|u(t)\|_{L^2}^2}{2|\mu(t)\mu_1^q(t)|}\left\|\psi_{\sigma}\left(\frac{{x}}{\mu_1(t)} \right)\right\|_{L^{\infty}}\left\|\phi_{\delta}'\left(\frac{{x}}{\mu_1^q(t)} \right)\right\|_{L^{\infty}}\\ 
      &\lesssim_{\sigma,\delta} \frac{\|u_0\|_{L^2}^2}{t^{1-b + bq}\log^{(2-q)}t},
\end{split}
\end{equation}
because of  $q > 1$ one has that $A_{1,3}(t) \in L^1(\{t \gg1\})$.
               
\medskip

Next we turn our attention to the other terms of (\ref{TermsEq.}). First, by means of Young's inequality, we have for $\epsilon > 0$
\begin{equation*}
\begin{split}
 |A_2(t)|
  &\leq \Big|\frac{\mu_1'(t)}{\mu_1(t)\mu(t)} \Big|\int_{\mathbb{R}}\Big|\psi_{\sigma}'\Big( \frac{{x}}{\mu_1(t)} \Big) \phi_{\delta}\Big( \frac{{x}}{\mu_1^q(t)} \Big)\Big|\Big[ \frac{u^2}{4\epsilon} + 4\epsilon\Big|\frac{{x}}{\mu_1(t)}\Big|^2\Big]\,dx    \\
  &\leq \frac{1}{4\epsilon}\Big|\frac{\mu_1'(t)}{\mu_1(t)\mu(t)} \Big|\int_{\mathbb{R}} u^2(x,t)\psi_{\sigma}'\Big( \frac{{x}}{\mu_1(t)} \Big) \phi_{\delta}\Big( \frac{{x}}{\mu_1^q(t)} \Big)\,dx\\
  &\hskip15pt + 4\epsilon\Big|\frac{\mu_1'(t)}{\mu_1(t)\mu(t)} \Big| \,\|\phi_{\delta}\Big( \frac{\cdot}{\mu_1^q(t)} \Big)\|_{L^{\infty}}
  \|\Big(\frac{\cdot}{\mu_1^q(t)}\Big)^2\psi_{\sigma}'\Big( \frac{\cdot}{\mu_1^q(t)}\Big)\|_{L^1}.
    \end{split}
    \end{equation*}

Then, taking $\epsilon = \mu_1'(t)$, which is positive in $\{t\gg 1 \} $, we get
    \begin{equation}\label{l42-10}
     \begin{split}
        |A_2(t)| & \leq \Big|\frac{1}{4\mu_1(t)\mu(t)} \Big|\int_{\mathbb{R}}u^2\psi_{\sigma}'\left( \frac{{x}}{\mu_1(t)} \right) \phi_{\delta}\left( \frac{{x}}{\mu_1^q(t)} \right)dx \\
        &\hskip15pt+C_{\delta,\sigma}\frac{(b\log t-1)^2}{t^{3-3b}\log^6t}\\
        &= \frac{1}{2}A_{1,2}(t) +C_{\delta,\sigma}\frac{(b\log t-1)^2}{t^{3-3b}\log^6t},
     \end{split}
    \end{equation}
where $C_{\sigma,\delta}$ is a constant depending on $\sigma$ and $\delta$.

Notice that the last term in the last inequality of \eqref{l42-10} is integrable in ${t \gg 1}$ since $b < \frac{2}{3}$.

Finally, we consider the term $A_3(t)$. Young's inequality and the conservation of mass tell us that
    \begin{equation}\label{l42-11}
    \begin{split}
     |A_3(t)|&\leq \left| \frac{q\mu_1'(t)}{\mu_1(t)\mu(t)}\right|\|\psi_{\sigma}\|_{L^{\infty}}\int_{\mathbb{R}} t^{1-b} u^2(x,t)\,dx\\
      &\hskip15pt +\left| \frac{q\mu_1'(t)}{\mu_1(t)\mu(t)}\right| \|\psi_{\sigma}\|_{L^{\infty}}\!\!\int_{\mathbb{R}}\frac{1}{t^{1-b}}\left[\left( \frac{{x}}{\mu_1^q(t)} \right) \phi_{\delta}'\left( \frac{x}{\mu_1^q(t)} \right)\right]^2\,dx\\
     & \lesssim_{\sigma,\delta} \left| \frac{qt^{1-b}\mu_1'(t)}{\mu_1(t)\mu(t)}\right| + \left| \frac{q\mu_1'(t)\mu_1^q(t)}{t^{1-b}\mu_1(t)\mu(t)}\right|.
    \end{split}
    \end{equation}
    
Hence, the conditions on \eqref{Relationderivatives} imply 
\begin{equation}\label{l42-12}
|A_{3}(t)| \lesssim_{\sigma,\delta} \frac{1}{t\log^2t} + \frac{1}{t^{3-b(2+q)}\log^{2+q}t}.
\end{equation}
Since $b \leq \frac{2}{2+q} $, $A_3(t) \in L^1(\{t \gg 1\})$.

Gathering the information in \eqref{l42-1}, \eqref{TermsEq.}, \eqref{l42-4}, \eqref{l42-9}, \eqref{l42-10}, \eqref{l42-11} 
and \eqref{l42-12} together we conclude that
    \begin{equation}\label{l42-13}
    \begin{split}
\frac{1}{2\mu(t)\mu_1(t)}\int_{\mathbb{R}} u^2(x,t)\psi_{\sigma}'\left(\frac{{x}}{\mu_1(t)} \right) \phi_{\delta}\left(\frac{{x}}{\mu_1^q(t)} \right)\,dx\leq & \frac{d}{dt}\mathcal{I}(t)+h(t)
    \end{split}
    \end{equation}
where $h(t)\in L^1(\{t\gg 1\})$, as desired.
\end{proof}

The next lemma will give us a key bound in our analysis.

\begin{lemma}\label{intl2bounded}
 Assume that $u_0 \in L^2(\R)$. Let $u \in C(\mathbb{R}: L^2(\mathbb{R}))\cap L^{\infty}(\mathbb{R}: L^2(\mathbb{R})) $ be the solution of the IVP \eqref{BO}. Then, there exists a constant $0< C < \infty$, such that 
 \begin{equation}
  \int_{ \{t \gg 1\} } \frac{1}{t\log t}\int_{B_{t^b}} u^2(x,t)\, dxdt \leq C. 
 \end{equation}

\end{lemma}

\begin{proof}
 From the definition, $\mu(t)\mu_1(t)=t\log t$ and a straightforward computation involving the properties of the function $\phi$, it follows that
    \begin{equation*}
    \frac{1}{\mu_1(t)\mu(t)}\int_{B_{t^b}}u^2(x,t)\,dx \leq \frac{1}{\mu_1(t)\mu(t)}\int_{\mathbb{R}}u^2\psi_{\sigma}'\left( \frac{{x}}{\mu_1(t)}\right)\phi_{\delta}\left( \frac{{x}}{\mu_1^q(t)}\right) \,dx, 
    \end{equation*}
for suitable $\sigma$ and $\delta$, whenever $q> 1$ is chosen sufficiently close to 1 and $b$ slightly smaller if necessary. Lemma \ref{BoundL1} implies that
\begin{equation}\label{l43-1}
 \int_{\{t \gg 1\}}\frac{1}{\mu_1(t)\mu(t)}\int_{B_{t^b}}u^2(x,t)\,dxdt \leq -\mathcal{I}(t) + \int_{\{t \gg 1\}}|h(t)|\,dt.
\end{equation}
The first term on the right hand side of inequality \eqref{l43-1} is bounded because of  $b \leq \frac{2}{2 + q} < \frac{2}{3}$ and the last one is bounded by the proof of Lemma \ref{BoundL1}. This
completes the proof of the lemma. 
 
\end{proof}

Now we are ready to prove Theorem \ref{L2BO}.
        
\subsection{Proof of Theorem  \ref{L2BO}}

Since the function $\frac{1}{t\log t} \notin L^1(B^c_r(1))$, from  the previous lemma, we can ensure that there exists a sequence $(t_n) \to \infty$, such that 
\begin{equation*}
 \lim_{n \to \infty}\int_{B_{(t_n)}^b}u^2(x, t_n)\,dx = 0.
\end{equation*}
Therefore, $0$ is an accumulation point and using that $u^2 \geq 0$ we can conclude the result.

%\end{proof}

%%%%%%%%%%%%%%%%%%%%%%%%%%%%%%%%%%%%%%%%%%%%%%%%%%%%%%%%%%%%%%%%%%%%%%%%%%%%%%%%%%%%%%%%%%%%%%%%%%%%%%%%%%%%%%%%%%%%%%%%

To end this section we will give a sketch of the proof of Corollary \ref{corL2BO}.

\subsection{Proof of Corollary \ref{corL2BO}}
The proof of this result follows the same argument as the proof given to prove Theorem \ref{L2BO}. Hence
we will present the new details introduced in the proof. We consider
the functional
\begin{equation*}
\mathcal{I}_{\rho}(t)=\frac{1}{\mu(t)}\int u(x,t)\,\psi_{\sigma}\left(\frac{x-\rho(t)}{\mu_1(t)}\right)\phi_{\sigma}\left(\frac{x-\rho(t)}{\mu_1^q(t)}\right)\,dx
\end{equation*}
where $\rho(t)=\pm t^m$, $m$ as in the statement of the corollary, $\mu(t)$ and $\mu_1(t)$ defined as in \eqref{parameters}, and
$\psi_{\sigma}$ and $\phi_{\delta}$ defined as above.

\medskip 

As in Lemma \ref{Functionalbounded} we have that
\begin{equation*}
\sup_{t\gg 1} | \mathcal{I}_{\rho}(t)| <\infty.
\end{equation*}

We also obtain a similar inequality as \eqref{NC-2} in Lemma \ref{BoundL1}, i.e.
\begin{equation*}
\begin{aligned}
& \frac{1}{\mu_1(t) \mu(t)}\int u^2(x,t)\,\psi_{\sigma}'\left(\frac{x-\rho(t)}{\mu_1(t)}\right)\phi_{\sigma}\left(\frac{x-\rho(t)}{\mu_1^q(t)}\right)\,dx
\\
&\qquad  \le 4\frac{d}{dt}\mathcal{I}_{\rho}(t)+h_{\rho}(t),
\end{aligned} 
\end{equation*}
where $h_{\rho}(t)\in L^1(\{t\gg1\})$. Besides the terms previously handle in the proof of Lemma  \ref{BoundL1}, here we need to estimate two additional
new terms
\[
- \frac{\rho'(t)}{\mu_1(t)\mu(t)}\int_{\mathbb{R}} u(x,t) \psi_{\sigma}'\left( \frac{x-\rho(t)}{\mu_1(t)} \right) \phi_{\delta}\left( \frac{x-\rho(t)}{\mu_1^q(t)} \right)\,dx=A(t)
\]
and 
\[
- \frac{\rho'(t)}{\mu_1^q(t)\mu(t)}\int_{\mathbb{R}} u(x,t) \psi_{\sigma}\left( \frac{x-\rho(t)}{\mu_1(t)} \right) \phi_{\delta}'\left( \frac{x-\rho(t)}{\mu_1^q(t)} \right)\,dx=B(t).
\]

The Cauchy-Schwarz inequality and the mass conservation yield
\begin{equation*}
\begin{split}
|A(t)+B(t)| &\le \left|\frac{\rho'(t)}{\mu_1^{1/2}(t)\mu(t)} \right| \|u_0\|_{L^2}\|\psi_{\sigma}'\|_{L^2}  \|\phi_{\delta}\|_{L^\infty}\\
&\hskip10pt + \left|\frac{\rho'(t)}{\mu_1^q(t)\mu(t)}\right| \|u_0\|_{L^2} \|\psi_{\sigma}\|_{L^{\infty}}  \|\phi_{\delta}'\|_{L^2}\\
& \lesssim_{\sigma,\delta, m} \frac{1}{t^{(4-2m -b)/2 }\log^{3/2}t}+\frac{1}{t^{(4 -2b -2m + bq)/2}\log^{(4-q)/2}t}.
\end{split}
\end{equation*}

We observe that the first term in the last inequality is in $L^1(\{t\gg 1\})$ since $m \leq 1 - \frac{b}{2}$. Similarly,  the last term in last inequality  is also in 
$L^1(\{t\gg 1\})$ since $m \leq 1 - \frac{b}{2} < 1 - b\frac{1 - q}{2}$.

\medskip

From this point on the argument of proof to establish Theorem \ref{L2BO} can be applied to end the proof of Corollary \ref{corL2BO}.
%\end{proof}

%%%%%%%%%%%%%%%%%%%%%%%%%%%%%%%%%%%%%%%%%%%%%%%%%%%%%
%%%%%%%%%%%%%%%%%%%%%%%%%%%%%%%%%%%%%%%%%%%%%%%%%%%%%
%%%%%%%%%%%%%%%%%%%%%%%%%%%%%%%%%%%%%%%%%%%%%%%%%%%%%

\section{Proof of Theorem \ref{H1BO} (Asymptotic Behavior in $H^{\frac{1}{2}}(\mathbb{R})$)}\label{Sect:4}

In this section contains the proof of Theorem \ref{H1BO}. The argument follows closely what we did in the previous section. Thus, we will give only the main new ingredients in the proof.

\subsection{ Asymptotic Behavior of $\left\|D^{1/2} u(t)\right\|_{L^2}$}

\begin{lemma}\label{BOasymp}
Let u $\in C(\mathbb{R}: H^{\frac{1}{2}}(\mathbb{R}))\cap L^{\infty}(\mathbb{R}: H^{\frac{1}{2}}(\mathbb{R}))$ the solution
of the IVP \eqref{BO}. Then, there exists a constant $C> 0$ such that
\begin{equation}\label{l61-2}
 \int_{\{ t \gg 1 \}}\frac{1}{t\log t}\int_{B_{t^b}(0)} \left| D^{1/2}u(x,t) \right|^2\,dxdt \leq C.
\end{equation}

\end{lemma}

\begin{proof}
 Consider the functional 
    \begin{equation}\label{l61-3}
{\mathcal{J}}(t) := \frac{1}{\mu(t)}\int_{\mathbb{R}}u^2(x,t)\psi_{\sigma}\left(\frac{x}{\mu_1(t)} \right)\,  dx.
    \end{equation}
 where $\mu(t)$ and $\mu_1(t)$ were defined in \eqref{parameters}.

Differentiating \eqref{l61-3} yields
\begin{equation}\label{l61-4}
\begin{split}
     \frac{d}{dt}{\mathcal{J}}(t) &=- \frac{\mu'(t)}{\mu^2(t)}\int_{\mathbb{R}}u^2(x,t)\,\psi_{\sigma}\left(\frac{x}{\mu_1(t)} \right)  \,dx\\
      &\hskip10pt +\frac{2}{\mu(t)}\int_{\mathbb{R}}u(x,t)\partial_tu(x,t)\,\psi_{\sigma}\left(\frac{x}{\mu_1(t)} \right)  \, dx\\
     &\hskip10pt - \frac{\mu_1'(t)}{\mu(t)\mu_1(t)}\int_{\mathbb{R}}u^2(x,t)\,\phi_{\sigma}\left(\frac{x}{\mu_1(t)} \right)\left(\frac{x}{\mu_1(t)} \right)\, dx\\
     &= A(t)+B(t)+C(t).
\end{split}
\end{equation}
    
Combining the properties $\mu(t)$ and $\mu_1(t)$, the conservation of mass, and using \eqref{parameters}, it follows that
    \begin{equation}\label{l61-5}
    |A(t)|+|C(t)| \lesssim_{\sigma} \frac{\|u_0\|_{L^2}^2}{t^{2-b}\log^2t}.
    \end{equation}
Thus, the terms $A(t)$, $C(t)$ are integrable in $\{t \gg 1\}$.

%Next, we turn our attention into $\tilde{\mathcal{I}}_3(t)$. Thus,
%\begin{equation}\label{l61-6}
% \begin{aligned}
% |\tilde{\mathcal{I}}_{3}(t)| \leq& \Big| \frac{\mu_1'(t)}{\mu(t)\mu_1(t)} \Big| \|u(t)\|_{L^2}^2\Big\| \phi_{\sigma}\left( \frac{\cdot}{\mu_1(t)}\right)\left( \frac{\cdot}{\mu_1(t)}\right) \Big\|_{L^{\infty}} \\
% \lesssim_{\sigma} & \frac{\|u_0\|_{L^2}^2}{t^{(2-b)}\log^2(t)}.
% \end{aligned} 
%\end{equation}
%As before, we obtain $\tilde{\mathcal{I}}_{3} \in L^1(\{t \gg 1\})$. 

Regarding $B(t)$, we use the equation in \eqref{BO} and integrate by parts to write
\begin{equation}\label{l61-7}
\begin{split}
B(t) &= - \frac{2}{\mu(t)}\int_{\mathbb{R}}\partial_xu\mathcal{H}\partial_xu \,\psi_{\sigma}\left(\frac{x}{\mu_1(t)} \right)  dx \\
    &\hskip15pt - \frac{2}{\mu(t)\mu_1(t)}\int_{\mathbb{R}}u\mathcal{H}\partial_xu\,\phi_{\sigma}\left(\frac{x}{\mu_1(t)} \right)  dx \\
    &\hskip15pt + \frac{2}{3\mu(t)\mu_1(t)}\int_{\mathbb{R}}u^3\phi_{\sigma}\left(\frac{x}{\mu_1(t)} \right)  dx \\
    &= B_1(t) + B_2(t) + B_3(t).
    \end{split}
    \end{equation}
    
From Hilbert's transform properties, integrating by parts and Cauchy-Schwarz inequality we obtain
\begin{equation}\label{l61-8}
 \begin{split}
|B_1(t)|=& \left|- \frac{1}{\mu(t)}\int_{\mathbb{R}}u\partial_x\left[\mathcal{H},\psi_{\sigma}\left(\frac{\cdot}{\mu_1(t)} \right)\right]\partial_xu  \,dx \right|\\ 
\leq& \frac{1}{\mu(t)}\|u\|_{L^2}\left\|\partial_x\Big[\mathcal{H},\psi_{\sigma}\left(\frac{\cdot}{\mu_1(t)} \right)\Big]\partial_xu\right\|_{L^2}.
 \end{split}
\end{equation}

  Lemma \ref{CC} gives us
 \begin{equation}\label{l61-9}
 \begin{split}
 |B_1(t)|  & \leq \frac{1}{\mu(t)}\|u\|_{L^2}^2\left\|\partial_x^2\psi_{\sigma}\left(\frac{\cdot}{\mu_1(t)} \right)\right\|_{L^{\infty}}\lesssim_{\sigma} \frac{1}{t^{1 + b}},
    \end{split}
    \end{equation}
which belongs to $L^1(\{t \gg 1\})$.

To estimate $B_2(t)$, we apply Plancherel's identity to obtain
\begin{equation}\label{l61-10}
 \begin{split}
 B_2(t) &= - \frac{2}{\mu_1(t)\mu(t)}\int_{\mathbb{R}}uD^{1/2}\left[D^{1/2},\phi_{\sigma}\left(\frac{\cdot}{\mu_1(t)} \right)\right]u  \,dx \\ 
  &\hskip15pt - \frac{2}{\mu_1(t)\mu(t)}\int_{\mathbb{R}}\left(D^{1/2}u\right)^2\phi_{\sigma}\left(\frac{x}{\mu_1(t)} \right) dx \\
  &= B_{2,1}(t) + B_{2,2}.
 \end{split}
\end{equation}

Notice that $B_{2,2}(t)$ is the term we want to estimate.

\medskip

To bound the term $B_{2,1}(t)$ we use Cauchy-Schwarz's inequality,  the conservation of mass, \eqref{comm-d-1/2}, and properties of the Fourier transform
to deduce that
\begin{equation}\label{l61-11}
 \begin{split}
 |B_{2,2}(t)| &\leq \left| \frac{2}{\mu_1(t)\mu(t)}\right| \|u_0\|_{L^2}
 \left\|\widehat{\left(\partial_x\phi_{\sigma}\left(\frac{\cdot}{\mu_1(t)} \right)\right)}\right\|_{L^1}\\
& \lesssim_{\sigma} \frac{1}{t^{1 + b}}  \in L^1(\{t \gg 1\}). 
 \end{split}
\end{equation}

Finally, notice that by Lemma \ref{GNSinequality}
    \begin{equation}\label{l61-12}
    \begin{aligned}
    & \int_{\mathbb{R}}|u|^3 \phi_{\sigma}\left(\frac{x}{\mu_1(t)} \right)dx  \\
    &  \leq \sum_{n \in \mathbb{Z}}\int_{\mathbb{R}}(|u|\,\zeta_n)^3 \phi_{\sigma}\left(\frac{x}{\mu_1(t)} \right)dx  \\
    & \leq \sum_{n \in \mathbb{Z}}\|u\, \zeta_n\|^3_{L^3}\left(\sup_{x \in [n, n+ 1]} \phi_{\sigma}\left(\frac{x}{\mu_1(t)} \right)\right)  \\
    & \lesssim \sum_{n \in \mathbb{Z}}\|u\,\zeta_n\|^2_{L^2}\|D^{1/2}(u\,\zeta_n)\|_{L^2}\left(\sup_{x \in [n, n+ 1]} \phi_{\sigma}\left(\frac{x}{\mu_1(t)} \right)\right).  
    \end{aligned}
    \end{equation}
   
Moreover, by Lemma \ref{leibnizhomog.} and hypothesis,
    \begin{equation}\label{l61-13}
    \begin{aligned}
     \|D^{1/2}(u\zeta_n)\|_{L^2}& \lesssim \left\|D^{1/2}u(t)\right\|_{L^2}\|\zeta_n\|_{L^{\infty}} + \|u(t)\|_{L^2}\|D^{1/2}\zeta_n\|_{L^{\infty}} \\
     & \lesssim \|u(t)\|_{H^{1/2}(\mathbb{R})} \lesssim \|u\|_{L_t^{\infty}H^{1/2}}.
     \end{aligned}
     \end{equation}
  
Combining these estimates we deduce that
\begin{equation*}\label{l61-14}
 \int_{\mathbb{R}}|u(x,t)|^3 \phi_{\sigma}\left(\frac{x}{\mu_1(t)} \right)dx\\
  \lesssim  \sum_{n \in \mathbb{Z}}\|u\,\zeta_n\|^2_{L^2}\left(\sup_{x \in [n, n+ 1]} \phi_{\sigma}\left(\frac{x}{\mu_1(t)} \right)\right).
\end{equation*}
A similar analysis to that given in Lemma 4.1 in  \cite{MMPP} (see also \cite{KM}) yields
\begin{equation}\label{l61-15}
 \int_{\mathbb{R}}|u(x,t)|^3 \phi_{\sigma}\left(\frac{x}{\mu_1(t)} \right)dx \lesssim \int_{\mathbb{R}}|u(x,t)|^2 \phi_{\sigma}\left(\frac{x}{\mu_1(t)} \right)\,dx .
\end{equation}

Using the properties of the function $\phi$ in \eqref{weight} for suitable $\delta$ and $\sigma$  we can apply Lemma \ref{intl2bounded} to deduce
that $B_3(t) \in L^1(\{t \gg 1\})$. 

\medskip

Collection the information in \eqref{l61-4}, \eqref{l61-5}, \eqref{l61-9}, \eqref{l61-11} and \eqref{l61-15}
we deduce that
\begin{equation*}
\frac{1}{t\log t}\int_{B_{t^b}} |D^{1/2}u(x,t)|^2\,dxdt\le \frac{d}{dt}\mathcal{J}(t)+g(t),
\end{equation*}
where  $\mathcal{J}(t)$ is bounded and $g(t)\in L^1(\{t\gg 1\})$.

A similar analysis as the one implemented in the proof of Theorem \ref{L2BO} yields the desired result.

\end{proof}

The reminder of the proof of Theorem \ref{H1BO} uses a similar argument as the proof of Theorem \ref{L2BO}, so we will omit it.

\section{Proof of Theorems \ref{to-the-right} and Theorem \ref{to-the-left}}\label{Sect:5}

The proofs of Theorems \ref{to-the-right} and \ref{to-the-left} are based on the following virial identity.

\begin{lemma} 
\label{virial-BO}

Let $u\in C(\mathbb R:H^1(\mathbb R))$ be the global real solution of IVP \eqref{BO}. Then for any weighted function $\varphi=\varphi(x,t)$ with 
$$
\varphi \in C(\mathbb R: L^{\infty} \cap \dot H^4(\mathbb R))
$$
the following identity holds
\begin{equation}
\label{id1}
\begin{aligned}
\frac{d}{dt}\int u^2(x,t)\varphi(x,t)\,dx &= -\int u \partial_x\big[\mathcal H;\varphi\big]\partial_x u \,dx\\
&\hskip12pt-\int (D^{1/2}_xu)^{1/2}\partial_x\varphi \,dx\\
&\hskip12pt-\int uD^{1/2}_x \big[D^{1/2}_x;\partial_x\varphi\big]u \,dx \\
&\hskip12pt+\frac{2}{3}\int u^3\partial_x\varphi \,dx+\int u^2\partial_t\varphi \,dx\\
&\equiv A_1+A_2+A_3+A_4+A_5.
\end{aligned}
\end{equation}
 
  \end{lemma}
  
\begin{remark}\label{rem-imp}
 The terms $A_1, A_2, A_3$ derive from the (linear) dispersive part of the equation. $A_2$ corresponds to the local smoothing effect of Kato type, first deduced in solutions of the KdV equation \cite{Ka}. As it was proved in Section 2 the terms $A_1, A_3$ are of order zero on $u$ and of order two, in the homogeneous sense,  on the weighted function $\varphi$.

\end{remark}

 \begin{proof} [Proof of Lemma \ref{virial-BO}]
 Using the equation we get
 \begin{equation}
 \label{eq1}
 \begin{aligned}
  \frac{d}{dt}\int &u^2(x,t)\varphi(x,t)\, dx =2\int u\partial_tu\varphi \, dx+\int u^2\partial_t\varphi \,dx\\
&=2\int  u(\mathcal H\partial_x^2u-u\partial_xu)\varphi \,dx +\int u^2\partial_t\varphi \,dx\\
&=2\int u\mathcal H\partial_x^2u \varphi \,dx+\frac{2}{3} \int u^3 \partial_x\varphi \,dx +\int u^2\partial_t\varphi \,dx\\
&=2\int u\mathcal H\partial_x^2u \varphi \,dx+A_4+A_5.
\end{aligned}
\end{equation}
By integration by parts it follows that
  \begin{equation}
 \label{eq2}
2\int u\mathcal H\partial_x^2u \varphi \,dx=-2\int \partial_xu \mathcal H\partial_xu \varphi \,dx -2\int u \mathcal H\partial_xu \partial_x\varphi \,dx.
\end{equation}
Since
  \begin{equation}
 \label{eq3}
 \begin{aligned}
 -\int \partial_xu \mathcal H\partial_xu \varphi \,dx&= \int \partial_xu \mathcal H\left(\partial_xu \varphi\right)\,dx\\
 &=\int \partial_xu \mathcal H\partial_xu \varphi \,dx +\int \partial_xu \big[\mathcal H;\varphi\big]\partial_xu\,dx,
 \end{aligned}
 \end{equation}
 one has that
 \begin{equation}
 \label{eq4}
 \begin{aligned}
-2\int \partial_xu \mathcal H\partial_xu \varphi \,dx &=\int \partial_xu \big[\mathcal H;\varphi\big]\partial_xu\,dx\\
 &=- \int u \partial_x\big[\mathcal H;\varphi\big]\partial_xu\,dx=A_1.
  \end{aligned}
 \end{equation}
Also 
 \begin{equation}
 \label{eq5}
 \begin{split}
 -2\int u &\mathcal H\partial_xu \partial_x\varphi \,dx=-2\int u D_xu \partial_x\varphi \,dx\\
&= -2\int D_x^{1/2}u D^{1/2}_x(u\partial_x\varphi)\,dx\\
&= -2\int D_x^{1/2}u D^{1/2}_xu\partial_x\varphi \,dx
 -2\int \!D_x^{1/2}u \big[D^{1/2}_x;\partial_x\varphi\big]u\,dx\\
 &=-2\int D_x^{1/2}u D^{1/2}_xu\partial_x\varphi \,dx-2\int \!u D_x^{1/2}\big[D^{1/2}_x;\partial_x\varphi\big]u\,dx\\
 &=A_2+A_3.
  \end{split}
 \end{equation} 
 Inserting \eqref{eq4} and \eqref{eq5} in \eqref{eq2}, and this in \eqref{eq1} we obtain \eqref{id1}.

\end{proof}

\subsection{Proof of Theorem \ref{to-the-right}}

%\begin{proof}
  
  First, we fix
  \begin{equation}
  \label{weight1}
  \varphi(x,t)=\chi \left(\frac{x-c_1}{c_0 t}\right),
  \end{equation}
  with $\chi$ 
satisfying
\begin{equation}
\label{chi}
\begin{cases}
\chi\in C^\infty(\R), \quad 0\leq \chi \leq 1 \quad \hbox{ in}\quad \R, \\
\chi(s) \equiv 0 \hskip34pt \hbox{if}\quad s\leq 1, \quad  \chi(s) \equiv 1 \quad \hbox{if}\quad s\geq 2, \\
\chi' (s) >0, \hskip26pt \hbox{in}\quad (1,2),\\
|\chi^{(k)} (s)|\leq 2^k, \hskip5pt \hbox{in}\quad (-1,0), \quad k=1,2,3,
\end{cases}
\end{equation}
  and $c_0, c_1$ constants to be chosen latter. We observe that
  \begin{equation}
  \label{chi1}
  \partial_t \chi \left(\frac{x-c_1}{c_0 t}\right)=\chi' \left(\frac{x-c_1}{c_0 t}\right)\left(\frac{x-c_1}{c_0 t}\right)\frac{-1}{t}\leq \chi'(\cdot) \,\frac{-1}{t}
  \end{equation}
 and
 \begin{equation}
  \label{chi2}
  \partial_x \chi \left(\frac{x-c_1}{c_0 t}\right) =\chi'(\cdot)\,\frac{1}{c_0 t}
  \end{equation} 
  
  With the notation in \eqref{id1} and using the commutator estimates in Lemmas \ref{CC} and \ref{comm2} it follows that
  \begin{equation}
  \label{est1}
  \begin{aligned}
  |A_1| &\leq \frac{c}{c_0^2 t^2} \| \chi''\|_{L^\infty} \|u(t)\|^2_{L^2},\\
  A_2 & \leq 0,\\
  |A_4|&\leq  \frac{2\,\|u(t)\|_{\infty}}{3 c_0 t}\,\int u^2(x,t)\chi'\left(\frac{x-c_1}{c_0t}\right) \,dx,\\
  A_5&\leq \frac{-1}{t}\,\int u^2(x,t)\chi'\left(\frac{x-c_1}{c_0t}\right) \,dx,
  \end{aligned}
  \end{equation}
  and
  \begin{equation}
  \label{est2}
  \begin{aligned}
  |A_3|&\leq \left(\frac{1}{c_0t}\right)^{5/2}\| \chi''(\cdot)\|_{L^2}^{1/2}\| \chi'''(\cdot)\|_{L^2}^{1/2}\|u(t)\|^2_{L^2}\\
  &\leq c_{\chi}\left(\frac{1}{c_0t}\right)^{2}\|u(t)\|^2_{L^2}.
  \end{aligned}
  \end{equation}
  
  Inserting the above estimates in \eqref{id1} and using  the conservation laws of the BO equation one finds that
  \begin{equation}
\label{est3}
\begin{aligned}
& \frac{d}{dt}\int u^2(x,t)\chi\left(\frac{x-c_1}{c_0 t}\right)\,dx \\
& \leq  \frac{c_{\chi}}{(c_0 t)^{2}}\|u_0\|_{L^2}^2\\
&\;\; + \left( \frac{2 c_2 \|u_0\|_{H^1}}{3 c_0 t} - \frac{1}{t}\right) \int u^2(x,t)\chi'\left(\frac{x-c_1}{c_0t}\right) \,dx 
   \end{aligned}
  \end{equation}
with $c_2$ a universal constant. Thus, we take $c_0$ such that
 \[
 \frac{2 c_2 \| u_0\|_{H^1}}{3 c_0} <1,
 \]
 and for any given  $\epsilon>0$ we fix $t_1>0$ such that
 \[
 \|u_0\|^2_2\;\int_{t_1}^{\infty}\frac{c_{\chi}}{(c_0 t)^{2}} \,dt\leq \epsilon.
 \]
 Hence, integration \eqref{est3} in the time interval $[t_1,t_2]$ we find that
 \begin{equation}
\label{est4}
 \int u^2(x,t_2)\chi\left(\frac{x-c_1}{c_0 t_2}\right)\,dx \leq \int u^2(x,t_1)\chi\left(\frac{x-c_1}{c_0 t_1}\right)\,dx + \epsilon
 \end{equation}
 Next, we fix $c_1>0 $ such that
 $$
 \int u^2(x,t_1)\chi\left(\frac{x-c_1}{c_0 t_1}\right)\,dx \leq \epsilon,
$$
 to get that for any $t_2>t_1$
 $$
   \int_{x>c_1+2c_0t_2} u^2(x,t_2)\,dx \leq \int u^2(x,t_2)\chi\left(\frac{x-c_1}{c_0 t_2}\right)\,dx +\epsilon \leq 2\epsilon.
 $$
 Finally, fixing $C_0=3c_0$ we have that
 $$
 \limsup_{t\to \infty}  \int_{x>C_0t} u^2(x,t)\,dx \leq 2\epsilon.
 $$ 
 which yields the desired result \eqref{main}.
  %\end{proof}

\subsection{Proof of Theorem \ref{to-the-left}}

 First we fix
  \begin{equation}
  \label{weight2}
  \varphi(x,t)=\beta \left(\frac{x+c_1}{\mu(t)}\right),
  \end{equation}
  with $\beta$ 
satisfying
\begin{equation}
\label{phi}
\begin{cases}
\beta\in C^\infty(\R), \quad 0\leq \beta \leq 1 \quad \hbox{ in}\quad \R, \\
\beta(s) \equiv 1 \hskip35pt \hbox{if}\quad s\leq -2, \quad  \beta(s) \equiv 0 \quad \hbox{if}\quad s\geq -1, \\
\beta' (s) <0, \hskip26pt \hbox{in}\quad (-2,-1),\\
|\beta^{(k)} (s)|\leq 2^k, \hskip5pt\hbox{in}\quad (-2,-1), \quad k=1,2,3,
\end{cases}
\end{equation}
with
 \begin{equation}
 \label{mu-def}
 \mu(t)=c_2t\,\log^{1+\eta}t,\;\;\;\;\;\eta>0,
 \end{equation}
 $c_1$ a constant to be chosen latter and $c_2>0$ an arbitrary constant. We observe that
  \begin{equation}
  \label{phi1}
  \partial_t \beta \left(\frac{x+c_1}{\mu(t)}\right)=\beta' \left(\frac{x+c_1}{\mu(t)}\right)\left(\frac{x+c_1}{\mu(t)}\right)\frac{-\mu'(t)}{\mu(t)}\leq 0,
  \end{equation}
 and
 \begin{equation}
  \label{phi20}
  \partial_x \beta \left(\frac{x+c_1}{\mu(t)}\right)=\beta' \left(\frac{x+c_1}{\mu(t)}\right)\,\frac{1}{\mu(t)}.
  \end{equation} 

 With the notation in \eqref{id1} from commutator estimates in Lemma \ref{CC} and Lemma \ref{comm2} it follows that
  \begin{equation}
  \label{est21}
  \begin{aligned}
  |A_1| &\leq \frac{c}{\mu^2(t)} \| \beta''\|_{L^\infty} \|u(t)\|^2_{L^2}\leq \frac{c_{\beta}}{\mu^2(t)}  \|u(t)\|^2_{L^2},\\
  A_2 & \leq \frac{c}{\mu(t)}  \| \beta'\|_{L^\infty} \left\|D_x^{1/2}u(t)\right\|_{L^2}^2\leq  \frac{c_{\beta}}{\mu(t)}   \left\|D_x^{1/2}u(t)\right\|_{L^2}^2,\\
   |A_3|&\leq \frac{c}{\mu^{5/2}(t)}\| \beta''(\cdot)\|_{L^2}^{1/2}\| \beta'''(\cdot)\|_{L^2}^{1/2}\|u(t)\|^2_{L^2}\\
  &\leq \frac{c_{\beta}}{\mu^{2}(t)}\,\|u(t)\|^2_{L^2},\\
  |A_4|&\leq  \frac{2\,\|u(t)\|_{L^3}^3\,\|\beta'\|_{L^\infty}}{3 \mu(t)}\leq  \frac{c_{\beta}\,\|u(t)\|_{L^\infty}\|u(t)\|_{L^2}^2}{\mu(t)},\\
  A_5&\leq 0.
  \end{aligned}
  \end{equation}
  
Inserting \eqref{est21} into the virial identity \eqref{id1} and using the conservation laws of the BO equation one sees that there exists $$K_0=K_0\left( \beta;\|u_0\|_{L^2};\|D_x^{1/2}u_0\|_{L^2};\|\partial_xu_0\|_{L^2} \right)>0$$ such that
\begin{equation}
\label{est22}
\frac{d}{dt}\int u^2(x,t) \beta \left(\frac{x+c_1}{\mu(t)}\right)\,dx\leq \frac{K_0}{t\,\log^{1+\eta}t}.
\end{equation}

Thus, given any $\epsilon>0$ we take $t_0>1$ such that 
\begin{equation}
\label{t_0}
\int_{t_0}^{\infty} \frac{K_0}{c_2 t\,\log^{1+\eta}t} dt\leq \epsilon
\end{equation}
%  \end{proof}
  to get that for any $t_1>t_0$
 \begin{equation}
\label{est23}
\begin{aligned}
&\int u^2(x,t_1) \beta \left(\frac{x+c_1}{c_2t_1\,\log^{1+\eta}t_1}\right)\,dx\\
&\qquad \leq \int u^2(x,t_0) \beta \left(\frac{x+c_1}{c_2t_0\,\log^{1+\eta}t_0}\right)\,dx +\epsilon.
\end{aligned}
\end{equation} 
 By taking  $c_1$ such that
\[
\int u^2(x,t_0) \beta \left(\frac{x+c_1}{c_2 t_0\,\log^{1+\eta}t_0}\right)\,dx <\epsilon,
\]
one finds that for any $t_1>t_0$
   \begin{equation}
\label{est24}
\begin{aligned}
&\int_{x<-c_1-2c_2t_t\,\log^{1+\eta}t_1} \!\!\! u^2(x,t_1) \,dx\\
&\qquad \leq \int u^2(x,t_1) \beta \left(\frac{x+c_1}{c_2t_1\,\log^{1+\eta}t_1}\right)\,dx\leq 2\epsilon.
\end{aligned}
\end{equation} 
Hence,
\begin{equation}
\label{est25}
\limsup_{t\to \infty}\int_{x<-3c_2t\,\log^{1+\eta}t} u^2(x,t) \,dx\leq 2\epsilon.
\end{equation}

 Since $\epsilon>0$ and $c_2>0$ are arbitrary we finish the proof.

\appendix

\section{Some auxiliary results}\label{Sect:6}

 \begin{proof} [Proof of Lemma \ref{comm2}] One sees that
 
 \begin{equation}
 \label{comm-proof1}
 \begin{aligned}
 &\widehat{\left(D^{1/2} \big [D^{1/2}; a\big] f\right)}(\xi)\\
 &=|\xi|^{1/2}\int |\xi|^{1/2}\left(\widehat{a}(\xi-\eta)\widehat{f}(\eta)-\widehat{a}(\xi-\eta)|\eta|^{1/2}\widehat{f}(\eta)\right)d\eta.
 \end{aligned}
 \end{equation}
 Therefore
 \begin{equation}
 \label{comm-proof2}
 \begin{aligned}
 &\big|\widehat{\left(D^{1/2} \big [D^{1/2}; a\big] f\right)}(\xi)\big|\\
 &\leq \int |\xi|^{1/2}\,\left| |\xi|^{1/2}-|\eta|^{1/2}\right| \left| \widehat{a}(\xi-\eta)\widehat{f}(\eta) \right| d \eta.
  \end{aligned}
 \end{equation}
 Assuming the following claim : 
 \begin{equation}
 \label{claim}
\exists \,c>0 \;\,s.t.\;\, \forall \,\xi, \, \eta\in\mathbb R\;\;\;\;\;
  |\xi|^{1/2}\big| |\xi|^{1/2}-|\eta|^{1/2}\big| \leq c|\xi-\eta|,
 \end{equation}
we shall conclude the proof.  From \eqref{claim} it follows that
\begin{equation}
 \label{comm-proof3}
 \begin{split}
 E_1(\xi)&\equiv \big|\widehat{\left(D^{1/2} \big [D^{1/2}; a\big] f\right)}(\xi)\big|\\
& \leq c \int |\xi-\eta| |\widehat{a}(\xi-\eta)\widehat{f}(\eta)| d \eta= c \int |\widehat{a'}(\xi-\eta)\widehat{f}(\eta)| d \eta.
  \end{split}
 \end{equation}
 Thus,
 \begin{equation}
 \label{comm-proof4}
\|E_1\|_2=\| \widehat{a'}\ast \widehat{f}\|_{L^2} \leq \| \widehat{a'}\|_{L^1} \|f\|_{L^2}.
\end{equation}
Using that
\begin{equation}
 \label{comm-proof5}
 \begin{split}
\| \widehat{a'}\|_1 & =\int_{|\xi|\leq R}|\widehat{a'}| d\xi + \int_{|\xi|>R} \frac{|\xi| |\widehat{a'}|}{|\xi|}\,d\xi\\
&\leq c R^{1/2} \| \widehat{a'}\|_{L^2} + c R^{-1/2}  \| \widehat{a''}\|_{L^2}.
 \end{split}
 \end{equation}
Choosing $R= \| \widehat{a''}\|_{L^2}^{1/2}/ \| \widehat{a'}\|_{L^2}^{1/2}$ we obtain \eqref{comm-d-1/2}. 

It remains to proof the claim in \eqref{claim}. First, we consider the case where $\xi$ and $\eta$ have the same sign, so we assume $\xi, \,\eta>0$. In this setting one sees that for some $\theta \in (0,1)$
\begin{equation}\label{1234}
\big| \xi^{1/2}-\eta^{1/2}\big|=\frac{1}{(\theta\xi+(1-\theta)\eta)^{1/2}}\,|\xi-\eta|.
\end{equation}
Thus, if $0<\xi/10<\eta$, \eqref{1234} yields the estimate in \eqref{claim}. 

If $0<\eta\leq \xi/10,$ one has
\[
\xi^{1/2}\left(\xi^{1/2}-\eta^{1/2}\right) \leq \xi\leq 2|\xi-\eta|.
\]

In the case where $\xi, \eta$ have different signs, one sees that
$$
|\xi-\eta|=|\xi|+|\eta|,
$$ and the estimate \eqref{claim} holds.

\end{proof}

\vspace{0.5cm}
\noindent{\bf Acknowledgements.}  R.F. was partially supported by CAPES, F.L. was partially supported by CNPq grant 305791/2018-4 and FAPERJ grant E-26/202.638/2019. C.M.'s work was funded in part by Chilean research grants FONDECYT 1191412, MathAmSud EEQUADD II, and Centro de Modelamiento Matem\'atico (CMM), ACE210010 and FB210005, BASAL funds for centers of excellence from ANID-Chile.

\medskip

%%%%%%%%%%%%%%%%%%%%%%%%%%%%%%%%%%%%%%%%%%%%%%%%%%%%%%%%%%%%%%%%%%%%%%%%%%%%%%%%%%%%%%%%%%%%%%%%%%%%%%%%%%%%%%%%%%%%%%%%%%%%%%%%%%%%%%%%%%%%%%%%%%%%%%%%%%%%%%%%%%%%%%%%%%%%%%%%%%%%%%%%%%%%%%%%%%%%%%%%%%%%%%%%%%%%%%%%%%%%%%%%%%%%%%%%%%%%%%%%%%%%%%%%%%%%%%%%%%%%%%%%%%%%%%%%%%%%%%%%%%%%%%%%%%%%%%%%%%%%%%%%%%%%%%%%%%%%%%%%%%%%%%%%%%%%%%%%%%%%%%%%%%%%%%%%%%%%%%%%%%%%%%%%%%%%%%%%%%%%%%%%%%%%%%%%%%%%%%%%%%%%%%%%%%%%%%%%%%%%%%%%%%%%%%%%%%%%%%%%%%%%%%%%%%%%%%%%%%%%%%%%%%%%%%%%%%%%%%%%%%%%%%%%%%%%%%%%%%%%%%%%%%%%%%%%%%%%%%%%%%%%%%%%%%%%%%%%%%%%%%%%%%%%%

\end{document}